\documentclass[12pt]{article}
\usepackage{latexsym}
\usepackage{amsfonts}
\usepackage{amsmath}
\usepackage{amsmath,amssymb,latexsym}

\makeindex

\newcommand{\forces}{\!\Vdash\!}
\newcommand{\imp}{\!\rightarrow\!}
\newcommand{\imps}{\!\rightarrow\!}

\newcommand{\pa}{{\sf PA}}
\newcommand{\gl}{{\sf GL}}
\newcommand{\gla}{{\sf GLA}}

\newcommand{\glacs}{{\sf GLA}_{\mbox{\it\tiny CS}}}

\newcommand{\glae}{{\sf GLA}_{\tiny\emptyset}}

\newcommand{\sfour}{{\sf S4}}

\newcommand{\lp}{{\sf LP}}

\newcommand{\ai}[1]{ {#1}^{\ast} }
\newcommand{\Ai}[1]{ ({#1})^{\ast} }

\newcommand{\proves}{\vdash}

\newcommand{\lc}[1]{#1\!\!:\!\!}

\newcommand{\co}[1]{! {#1} }

\newcommand{\Provable}[1]{ \hbox{\it Provable\/}({#1}) }
\newcommand{\Proof}[1]{ \hbox{\it Proof\/}({#1}) }

\newcommand{\PA}{{\sf PA}}

\newtheorem{Prop}{\bf Proposition}
\newenvironment{proposition}{\begin{Prop}\em }{\end{Prop}}
\newtheorem{Theor}{\bf Theorem}
\newenvironment{theorem}{\begin{Theor}\em }{\end{Theor}}
\newtheorem{Lemma}{\bf Lemma}

\newtheorem{Coro}{\bf Corollary}

\newtheorem{Fact}{\bf Fact.}

\newtheorem{Remark}{\bf Remark}

\newtheorem{Claim}[enumi]{Claim}

\newtheorem{defin}{\bf Definition}

\newtheorem{exam}{\bf Example}

\newtheorem{notat}{\bf Notation.}

\newenvironment{proof}{{\bf Proof.}}{\hfill $\slot$}
\newcommand{\slot}{\hfill \mbox{$\Box$}\vspace{\parskip}\\}
\newtheorem{Comment}{\bf Comment}

\begin{document}

\title{On Logic of Formal Provability and Explicit Proofs}

\author{Elena Nogina\thanks{Supported by PSC CUNY Research Awards program.} \\ \\
 {\small BMCC CUNY, Department of Mathematics}\\
{\small 199 Chambers Street, New York, NY 10007}\\
{\small {\tt E.Nogina@gmail.com}} }
\date{\empty}
\maketitle

\begin{abstract}  
In 1933, G\"odel considered two modal approaches to describing provability. One captured formal provability and resulted in the logic GL and Solovay's Completeness Theorem. The other was based on the modal logic S4 and led to Artemov's Logic of Proofs LP. In this paper, we study introduced by the author logic {\sf GLA}, which is a fusion of {\sf GL} and {\sf LP} in the union of their languages. {\sf GLA} is supplied with a Kripke-style semantics and the corresponding completeness theorem. Soundness and completeness of \gla\ with respect to the arithmetical provability semantics is established.  
\end{abstract}

\medskip\par
\section{Introduction}

G\"odel in \cite{God33} suggested a provability reading of modal logic {\sf S4}, which is axiomatized over the classical logic by the following list of postulates:
\medskip\par
$\Box(F\imp G)\imp(\Box F\imp \Box G)$ \hfill {\em Deductive Closure/Normality}\par
$\Box F \imp \Box\Box F$ \hfill {\em Positive Introspection/Transitivity}\par
$\Box F \imp F$ \hfill {\em Reflection}\par
\noindent
and the {\em Necessitation Rule}:  $\ \vdash F\ \Rightarrow\ \vdash \Box F$.
\medskip\par
G\"odel considered the interpretation of $\Box F$ as the formal provability predicate 
\[ \mbox{\emph{$F$ is provable in Peano Arithmetic $\PA$}} \]
and noticed that this semantics is inconsistent with $\sfour$.

Indeed, ~$\Box(\Box F\imps F)$ can be derived in ${\sf S4}$. On the other hand, interpreting~$\Box$ as the
predicate ``{\it Provable}" of formal provability in Peano Arithmetic $\PA$ and $F$ as {\it falsum} $\bot$,
converts this formula into the false statement that the consistency of $\PA$ is internally provable in $\PA$:
$$
\Provable{\mbox{\it Consis $\PA$}}.
$$
\subsection{Formal provability spills over to non-standard proofs} 

Let $\Proof{x,F}$ be a standard proof predicate  (cf. \cite{AB05,Boo93,Fef60}) $\mbox{\it $x$ is a proof for $F$};$ {\it Provable F} be $\exists x\Proof{x,F}$. 

Peano Arithmetic ${\sf PA}$ cannot distinguish between standard and nonstandard numbers; given $\exists x\Proof{x,F}$, $x$ may be a nonstandard number, hence not a code of any derivation in $\PA$. It means that $\mbox{\it Provable F}\ \imp F$ can fail in a model, and hence is not derivable in {\sf PA}. 

Indeed, consider a theory $ \mbox{{\sf T} = {\sf PA} + $\mbox{\it Provable}\ \bot$.} $ {\sf T} is consistent, since ${\sf PA}$ does not prove $\neg\mbox{\it Provable}\ \bot$. Hence {\sf T} has a model $M$ in which $\mbox{\it Provable}\ \bot$ holds, but $\bot$ does not.

So, the formal provability interpretation of $\sfour$ does not work; a provability calculus was left without a semantics and a provability semantics was left without a calculus thus opening two problems: 
\begin{enumerate}
\item Find a precise provability semantics for ${\sf S4}$;
\item 
Find a modal logic of formal provability {\it Provable}. 
\end{enumerate}

Problem 2 was solved in~1976 by Solovay \cite{Sol76}, who proved the completeness of G\"odel-L\"ob logic {\sf GL} with respect to the formal provability in arithmetic {\sf PA}.

In 1995, Problem 1 found its solution in Artemov's Logic of Proofs $\lp$ which provided a semantics of explicit proofs for {\sf S4} (\cite{Art95,Art01a}).

\subsection{G\"odel-L\"ob logic of formal provability}

Logic of Formal Provability ${\sf GL}$ (standing for G\"odel-L\"ob)  is given by the following list of postulates:
\begin{enumerate}
\item \textit{Axioms and rules of classical propositional logic} 
\item $\Box(F\!\imp\! G)\imp(\Box F\!\imp\!\Box G)$ \hfill {\it Deductive Closure/Normality}
\item $\Box F\!\imp\!\Box\Box F$ \hfill {\it Verification/Transitivity}
\item $\Box(\Box F\!\imp\!F)\imps\Box F$  \hfill {\it L\"ob Axiom}
\item \textit{Necessitation Rule}: $\ \ \ \ \ \ \ \ {\displaystyle \frac{\vdash F}{\proves\Box F}}$ 
\end{enumerate}
\medskip\par
Formal provability interpretation of a modal language is a mapping $\ast$ from the set of modal
formulas to the set of arithmetical sentences such that $\ast$ agrees with Boolean connectives and constants and 
\[ \Ai{\Box G}  = \mbox{\it Provable}\ \ai{G}.\]
{\bf Solovay's completeness theorem} (\cite{Boo93,Sol76}): 
\[ \mbox{\it $\gl\vdash F\ \ \ $ iff $\ \ \ $ for all formal provability interpretations $\ast$, $\ \PA\proves\ai{F}$.} \]

In 1938, G\"odel outlined a way to provide a provability semantics for {\sf S4} (\cite{God38}): modality there should be read explicitly as proof assertions $\lc{t}F$ interpreted as
\[ \mbox{\it t is a proof of F in Peano Arithmetic $\PA$}. \]
This G\"odel's suggestion was realized in Artemov's Logic of Proofs (\cite{Art95,Art01a}). 

\subsection{Artemov's Logic of Proofs}
Proof terms in {\sf LP} are built from constants and variables by two binary operations {\em application} ``$\cdot$" and {\it sum} ``$+$", and one unary operation {\em proof checker} ``$!$". Formulas of {\sf LP} are built as the usual propositional formulas with an additional formation rule: whenever $F$ is a formula and $t$ a proof terms, $\lc{t}F$ is a formula. 
\medskip\par
Axioms and rules of the Logic of Proofs $\lp$ are those of classical propositional logic plus axioms 
\medskip\par
\par $\lc{s}(F\imp G)\ \imp\ (\lc{t}F\imp \lc{[s\!\cdot\!t]}G)$ \hfill \emph{Application}
\par  $\lc{t}F\ \imp\ \lc{\co{t}}(\lc{t}F)$ \hfill \emph{Proof Checker}
\par $\lc{s}F\imp\lc{[s\!+\!t]}F$, $\ \ \lc{t}F\imp\lc{[s\!+\!t]}F$ \hfill \emph{Sum}
\par $\lc{t}F\imp F$ \hfill \emph{Explicit Reflection} 
\medskip\par\noindent
Each axiom $A$ is assumed internally provable, which is represented by formula $\lc{c}A$ where $c$ is a proof constant. The fundamental property of {\sf LP} is  given by Artemov's Realization Theorem (\cite{Art95,Art01a}): {\em for each theorem F of {\sf S4} one could recover a witness (proof term) to each occurrence of $\Box$ in F in such a way that the resulting formula $F^r$ is derivable in {\sf LP}}.
This theorem embeds {\sf S4} into {\sf LP}.  Further interpretation of {\sf LP} proof terms as formal proofs in {\sf PA} (\cite{Art95,Art01a}) provided a G\"odelian provability semantics for {\sf LP} and {\sf S4} and completed G\"odel's project of 1933. Nowadays, the Logic of Proofs has evolved into a general logical theory of justification \cite{AF11,AN05b,AN05c}.

\subsection{Comparing two G\"odel approaches to provability}

Logic of formal provability $\gl$ formalizes G\"odel's second incompleteness theorem $$\neg\Box(\neg\Box\bot),$$ L\"ob's
theorem $$\Box(\Box F\!\imp\!F)\imps\Box F,$$ and a number of other meaningful provability principles. 

Logic of Proofs $\lp$ represents proofs explicitly, naturally extends typed $\lambda$-calculus,
modal logic, and modal $\lambda$-calculus. 

$\gl$ and~$\sfour/\lp$ complement each other by addressing different areas of application.
$\gl$ finds applications in traditional proof theory.
$\lp$ targets areas of mathematical theories of knowledge and justification, foundations of verification,  typed theories and lambda-calculi, etc.

\subsection{Mixture of provability and explicit proofs}
Certain principles require a mixture of both provability and explicit
proofs. Consider the negative introspection principle. Its purely modal formulation $\neg\Box F\imp\Box\neg\Box F$ is not valid as a provability
principle. Indeed, let $F$ be $\bot$. Then $\neg\Box\bot$ reads as {\it Consis {\sf PA}} and the whole formula as 
\[ \mbox{\it Consis {\sf PA}}\imp\Provable{\mbox{\it Consis {\sf PA}}}, \] which is false, by G\"odel's Second
Incompleteness Theorem. 

There is no explicit negative introspection either.  The principle $\neg\lc{p}S\imp \lc{t}(\neg\lc{p}S),$
where $p$ and $t$ are proof terms and $S$ is a propositional variable, is not valid. Indeed, fix an
interpretation $\ast$ of $p$ and $t$ and the standard G\"odel proof predicate. There are infinitely many arithmetical
instances of $S$ for which the antecedent holds. Hence $\ai{t}$ should be a proof of infinitely many theorems, which is
impossible. However, the mixed language of proofs and provability fits this version of negative introspection: 
$$\neg\lc{p}F\imp \Box(\neg\lc{p}F)$$ is arithmetically provable, by $\Sigma$-completeness of $\PA$, according to which for each $\Sigma$-formula $\sigma$, $$\PA\proves\sigma \imp \mbox{\it Provable}\ {\sigma}.$$

We develop introduced in \cite{Nog06} a joint logic of formal provability and explicit proofs \gla\ (G\"odel-L\"ob-Art\"emov logic) in the language with provability assertions $\Box F$ and proof assertions $\lc{t}F$, find Kripke semantics for \gla\ and establish the arithmetical completeness of this logic.  

{\sf GLA} proved to be useful for applications in formal epistemology where it became a template for a family of epistemic logics with justifications (cf. \cite{AN05b,AN05c}). An elaborate proof theory of \gla\ and another version of Kripke models for \gla\ were offered by Kurokawa in \cite{Kur12,Kur13}.

\section{Description and basic properties of \gla}

The following two systems are predecessors of \gla: 
\begin{itemize}
\item
system ${\sf B}$ from \cite{Art94}, which does not have operations on proofs; 
\item
system ${\sf LPP}$  from \cite{Sid97a,Y_S01} in an extension of languages of the logic of formal provability \gl\  and the Logic of Proofs \lp.
\end{itemize}
Immediate successors of \gla\ are the logic {\sf GrzA} of strong provability and explicit proofs \cite{Nog09}, and symmetric logic of proofs and provability \cite{Nog10}.
\medskip\par\noindent
{\bf Language of \gla.}
\medskip\par\noindent
\emph{Proof terms} are built from \emph{proof variables} $x,y,z,\dots$ and \emph{proof constants}
$a,b,c,\dots$ by means of two binary operations: \emph{application}~`$\cdot$' and \emph{union}~`$+$', and one
unary \emph{proof checker}~`$!$'.

Formulas of {\sf GLA} are defined by the grammar
\[ A=S\mid A\imp A \mid A\wedge A \mid A\vee A \mid \neg A \mid \Box A \mid \lc{t}A\ ,\]
where $t$ stands for any proof term and $S$ for any sentence letter.

Axioms and rules of both G\"odel-L\"ob logic {\sf GL} and {\sf LP}, together with three specific principles
connecting explicit proofs with formal provability, constitute $ \sf \glae$.

\medskip\par\noindent
I. {\bf Axioms of classical propositional logic} \medskip\par
Standard axioms of the classical logic (e.g., {A1-A10} from \cite{Kle52})
\medskip\par\noindent
II. {\bf Axioms of Provability Logic {\sf GL}}\medskip\par
{\bf GL1} $\Box(F\imp G)\imp(\Box F\imp \Box G)$ \hfill {\em Deductive Closure/Normality}\par
{\bf GL2} $\Box F \imp \Box\Box F$ \hfill {\em Positive Introspection/Transitivity}\par
{\bf GL3} $\Box (\Box F \imp F)\imp \Box F$ \hfill {\em L\"ob Principle}\par
\medskip\par\noindent
III. {\bf Axioms of the Logic of Proofs $\lp$}
\medskip\par {\bf LP1} $\lc{s}(F\imp G)\ \imp\ (\lc{t}F\imp \lc{[s\!\cdot\!t]}G)$ \hfill
\emph{Application}
\par {\bf LP2} $\lc{t}F\ \imp\ \lc{\co{t}}(\lc{t}F)$ \hfill \emph{Proof Checker}
\par {\bf LP3} $\lc{s}F\imp\lc{[s\!+\!t]}F$, $\ \ \lc{t}F\imp\lc{[s\!+\!t]}F$ \hfill \emph{Sum}
\par {\bf LP4} $\lc{t}F\imp F$ \hfill \emph{Explicit Reflection}
\medskip\par\noindent
IV. {\bf Axioms connecting explicit and formal provability}
\medskip\par {\bf C1} $\lc{t}F\imp \Box F$ \hfill \emph{Explicit-Implicit connection} \par
{\bf C2} $\neg\lc{t}F\imp \Box\neg\lc{t}F$ \hfill \emph{Explicit-Implicit Negative Introspection} \par
{\bf C3} $\lc{t}\Box F\imp F$ \hfill \emph{Explicit-Implicit Reflection} \par
\medskip\par\noindent
V.  {\bf Rules of inference}
\par {\bf R1} $F\imp G,\ F\proves G$ \hfill \emph{Modus Ponens} \par
{\bf R2} $\vdash F\ \Rightarrow\ \vdash \Box F$ \hfill \emph{Necessitation}\par
{\bf R3} $\vdash \Box F \Rightarrow\ \vdash F$ \hfill \emph{Reflection Rule}\par
\medskip\par\noindent
A {\bf Constant Specification} ${C\!S}$ for $\gla$ is the set of formulas 
\[ \{\lc{c_1}A_1,\lc{c_2}A_2,\lc{c_3}A_3,\ldots\}, \]
where each $A_i$ is an axiom of $\glae$ and each $c_i$ is a proof constant.
$$\glacs\ = \glae\ + C\!S,$$
$$\gla\ = \glacs\ \mbox{with the ``total" \it$C\!S$}. $$

\begin{theorem}\label{internalizationtheorem}{\bf (Internalization Theorem)}. \\
\emph{If $\gla\proves F$ then for some proof term $p$, $\gla\proves\lc{p}F$.}
\end{theorem}
\begin{proof}
Induction on a derivation of $F$. \par
Base: $F$ is an axiom. Then use Constant Specification. In this case, $p$ is a proof constant. 

Induction steps: by internalized rules of $\gla$. 
\medskip\par
Internalization of {\it Modus Ponens} immediately follows from the Application axiom {\bf LP1}. 
\medskip\par
Internalization of Necessitation rule $\vdash F\ \Rightarrow\ \vdash \Box F$: \\
{\it For each $F$ there is $t(x)$ such that $\gla\vdash\lc{x}F\imp \lc{t(x)}\Box F$}
\medskip\par
1. $\lc{x}F\imp\Box F$ - axiom Explicit-Implicit Connection {\bf C1};\par
2. $\lc{a}(\lc{x}F\imp\Box F)$ - , from 1, by Constant Specification; \par
3. $\lc{x}F\imp\ \lc{!x}\lc{x}F$ - axiom Proof Checker {\bf LP2}; \par
4. $\lc{!x}\lc{x}F\imp\lc{(a\cdot !x)}\Box F$ - from 2, by Application {\bf LP1};\par
5. $\lc{x}F\imp\lc{(a\cdot !x)}\Box F$ - from 3,4, by propositional logic. 
\medskip\par\noindent
Now put $t(x)=a\cdot!x$. 
\medskip\par
Internalization of Reflection rule $\vdash \Box F \Rightarrow\ \vdash F$  \\
{\it For each $F$ there is $s(x)$ such that $\gla\vdash\lc{x}\Box F\imp \lc{s(x)}F$} 
\medskip\par
1. $\lc{x}\Box F\imp F$ - axiom Explicit-Implicit Reflection {\bf C3};\par
2. $\lc{b}(\lc{x}\Box F\imp F)$ from 1, by Constant Specification;\par
3. $\lc{x}\Box F\imp\lc{!x}\lc{x}\Box F$ - Proof Checker {\bf LP2}; \par
4. $\lc{!x}\lc{x}\Box F\imp\lc{(b\cdot!x)}F$ - from 2, by Application {\bf LP1}; \par
5. $\lc{x}\Box F\imp\lc{(b\cdot!x)}F$ - from 3,4, by propositional logic.
\medskip\par\noindent
Now put $s(x)=b\cdot!x$.  Note that in 2, we need an {\it internalized} Explicit-Implicit Reflection! 

\end{proof}

The list of postulated axioms and rules of \gla\ contains some principles which are derivable from the rest of the system. Such redundancies are generally acceptable to make exposition more readable. For example, in \gla\ (as well as in the Provability Logic \gl) the positive introspection axiom {\bf GL2} is derivable from the rest of the system (cf. \cite{Boo93}). In \gla\ the same holds for Reflection Axiom {\bf LP4}, Necessitation Rule {\bf R2} and Reflection Rule {\bf R3}. 
In all these cases we decide to postulate the corresponding principles for the sake of more concise definitions of important
subsystems of \gla.

Note that for any finite constant specification $C\!S$ the rule of necessitation is not redundant in $\glacs$ since to emulate {\bf R2} one needs an infinite constant specifications.

Here is an example of a yet more delicate dependency in  \gla: even though Explicit-Implicit Reflection Axiom {\bf C3} is derivable from
the rest of $\glae$ (Proposition~\ref{C3} below), proof constants corresponding to {\bf C3} are needed to guarantee the Internalization Property of \gla\
(cf. Theorem~\ref{internalizationtheorem}). Hence, we keep {\bf C3} as a basic postulate of \gla.
\begin{proposition}\label{C3} \emph{{\bf C3} is derivable from the rest of $\glae$.}
\end{proposition}
\begin{proof} The following is a derivation of $\lc{t}\Box F\imp F$ in $\glae$ without {\bf C3}.
\begin{tabular}{lll}
&&\\
1. & $\neg\Box F\imp\neg\lc{t}\Box F$, & contrapositive of {\bf LP4}; \\
2. & $\neg\lc{t}\Box F\imp\Box(\neg\lc{t}\Box F)$, & axiom {\bf C2};\\
3. & $\Box(\neg\lc{t}\Box F)\imp\Box(\lc{t}\Box F\imp F)$, & by reasoning in {\sf GL};\\
4. & $\neg\Box F\imp\Box(\lc{t}\Box F\imp F)$, & from 1,2, and 3;\\
5. & $\Box F\imp\Box(\lc{t}\Box F\imp F)$, & by reasoning in {\sf GL};\\
6. & $\Box(\lc{t}\Box F\imp F)$, & from 4 and 5;\\
7. & $\lc{t}\Box F\imp F$, & by {\bf R3}.\\
\end{tabular}

\end{proof}

{\sf GLA} is closed under substitutions of proof terms for proof variables and formulas for propositional
variables, enjoys the deduction theorem, and contains both {\sf GL} and {\sf LP}.

\subsection{Some principles of \gla }
{\bf Positive Introspection:} $\ \ \gla\proves\lc{t}F\imp\Box\lc{t}F$\par
1. $\lc{t}F\imp\ \lc{!t}\lc{t}F$ - Proof Checker axiom {\bf LP2};\par
2. $\lc{!t}\lc{t}F\imp\Box\lc{t}F$ - Explicit-Implicit Connection axiom {\bf C1}; \par
3. $\lc{t}F\imp\Box\lc{t}F$ - from 1,2, by propositional logic. \\ \\
{\bf Stability of proof assertions:} $\ \ \gla\proves\Box\ \lc{t}F\vee\Box\neg\lc{t}F$ \par
4. $\neg\lc{t}F\imp\Box\neg\lc{t}F$ - Explicit-Implicit Negative Introspection {\bf C2}; \par
5. $\Box\ \lc{t}F\vee\Box\neg\lc{t}F$ - from 3,4, by propositional logic.\\ \\
{\bf Explicit version of L\"ob Principle}. 
In $\Box(\Box F\imp F)\imp \Box F$ both modalities of the depth 1 can be read explicitly as $$\lc{x}(\Box F\imp
F)\imp\lc{l(x)}F$$ for some proof term $l(x)$. Indeed, 
\medskip\par
1. $\lc{x}(\Box F\imp F)\imp \lc{t(x)}\Box(\Box F\imp F)$ - by Internalized Necessitation Rule; \par
2. $\lc{c}(\Box(\Box F\imp F)\imp \Box F)$ - from L\"ob Principle {\bf GL3} by Constant Specification; \par
3. $\lc{t(x)}\Box(\Box F\imp F)\imp \lc{(c\cdot t(x))}\Box F$ - from 1,2 by Application {\bf LP1}; \par
4. $\lc{(c\cdot t(x))}\Box F\imp\lc{s(c\cdot t(x))} F$ - by Internalized Reflection Rule; \par
5. $\lc{x}(\Box F\imp F)\imp\lc{s(c\cdot t(x))} F$ - from 1,3,4.
\\ \\ 
{\bf L\"ob Principle cannot be realized in full.} 
Suppose for some proof polynomials $u$ and $v$,
$$\gla\proves\lc{x}(\lc{u}\bot\imp\bot)\imp\lc{v}\bot,$$
 hence $\gla\proves\lc{x}(\lc{u}\bot\imp\bot)\imp \bot$ and so
$F=\neg\lc{x}(\lc{u}\bot\imp\bot)$ is derivable in $\gla$. Consider a $\gla$-derivable formula
$$G=\lc{c}(\lc{u}\bot\imp\bot).$$
 Let us perform a substitution $\tau=[c/x]$ to both $F$ and $G$.
Then $F$ becomes $\neg\lc{c}(\lc{\tau u}\bot\imp\bot)$ and $G$ yields $\lc{c}(\lc{\tau u}\bot\imp\bot)$,
which is impossible. 
\subsection{Realizable provability principles.}  A Franco Montagna's question which
theorems of $\gl$ are realizable in $\gla$, has been answered by Evan Goris in \cite{Gor07,Gor09}.

It follows from the realization theorem for $\lp$ that all formulas of $\gl\bigcap\sfour$ are
realizable in $\lp$, and the question was actually whether proof terms of $\gla$ were capable of realizing some
other modal theorems of $\gl$. Goris' Theorem yields that it is not the case.\\ \\
{\bf Theorem} \cite{Gor07,Gor09}. {\it Only those theorems of \gl\ are realizable in $\gla$ which are from $\sfour$. }

\section{Models for {\sf GLA}}\label{sectionmodels}
In this section, we build Kripke-style models for \gla, which were described in \cite{Nog07}. 

A \emph{frame} is a standard $\gl$-frame $(W,\prec,\mbox{\it root})$ with the root node {\it root},  where $W$ is a
non-empty set of \emph{possible worlds}, $\prec$ is a binary transitive and conversely well-founded
\emph{accessibility} relation on $W$ (a relation $\prec$ is conversely well-founded if any increasing chain
$a_1\prec a_2\prec a_3\prec\ldots$ is finite).

{\bf Possible evidence relation} (first considered by Mkrtychev and then by Fitting) is a relation $\cal E$ between proof terms and formulas such that the following {\it closure conditions} are met:
\\ \\
{\em Application}: ${\cal E}(s,F\imp G)$ and ${\cal E}(t,F)$ implies ${\cal E}(s\!\cdot\!t,G)$.\\
{\em Proof Checker}: ${\cal E}(t,F)$ implies ${\cal E}(!t,(\lc{t}F))$. \\
{\em Sum}: ${\cal E}(s,F)$ or ${\cal E}(t,F)$ implies ${\cal E}(s+t,F)$.
\\ 

{\bf Model} is a structure ${\cal M} = (W,\prec,\mbox{\it root},{\cal E},\ \forces\ )$; here $\ \forces\ $ is a
relation between worlds and formulas such that \\ \\
 1. $\forces$ respects Boolean connectives at each world \\
($u\forces F\wedge G$ iff $u\forces F$ and $u\forces G$;
$u\forces\neg F$ iff $u\not\forces F$, etc.); \\ \\
2. $u\forces \Box F$ iff $v\forces F$ for every $v\in W$ with $u\prec v$;\\ \\
3. $u\forces \lc{t}F\ \ \ \ $ iff $\ \ \ \ {\cal E}(t,F)$ and $v\forces F$ for every $v\in W$.
\medskip\par
Following Solovay, we define
$${\cal H}(F) = \{\Box G\imp G\mid\Box G\ \mbox{\em is a subformula of F}\};  $$
for a set of formulas $X$, 
$${\cal H}(X)=\bigcup_{F\in X} {\cal H}(F).$$
A model $\cal M$ is called \emph{$F$-sound} if $\mbox{\em root}\forces{\cal H}(F)$. For a set of formulas $X$, $\cal M$ is \emph{$X$-sound}  if $\cal M$ is $F$-sound for each $F\in X$.\\ \\
For a given constant specification $C\!S$, a model $\cal M$ is a \emph{$C\!S$-model} if $\cal M$ is $C\!S$-sound
and $C\!S$ holds in $\cal M$.

\begin{theorem}{\bf (Soundness)}~\label{soundness}  \emph{For any formula F and any constant specification CS, if F is derivable in $\glacs$ then F holds
in each F-sound $C\!S$-model.}
\end{theorem}

\begin{theorem}\label{completenestheorem} {\bf (Completeness) }\emph{For any finite constant
specification $C\!S$ if F is not derivable in $\glacs$, then there is an F-sound $C\!S$-model with a finite frame where F does not hold.}
\end{theorem}

Proof goes by a canonical model construction with the use of technique developed by Solovay \cite{Sol76}, Artemov \cite{Art94}, and Fitting \cite{Fit05}. $\glae$ exhibits some sort of a finite model property, which also yields the decidability of $\glacs$ for any given finite constant specification:

\begin{theorem} \emph{For any finite constant specification $C\!S$, the logic $\glacs$ is decidable.}
\end{theorem}

\section{Provability semantics for {\sf GLA}, completeness}\label{provabilitysemantics}

In what follows, all proof predicates are assumed \emph{normal} (\cite{Art01a}), i.e., satisfying two properties. 
\medskip\par
1. Finiteness of proofs. \par 
For every $k$ set $T(k)=\{\varphi\mid\Proof{k, \varphi}\}$ is finite, the function from
$k$ to $T(k)$ is computable.
\medskip\par
2. Conjoinability of proofs.  \par For any $k$ and $l$ there is $n$ such that
\[ T(k)\cup T(l)\subseteq T(n) .\]
Prime example: G\"odel's proof predicate. 
\medskip\par
{\bf Arithmetical interpretation of \gla} is the sum of the intended arithmetical interpretations for $\gl$ and $\lp$. In particular, 
\[ \Ai{\Box G}  =  \mbox{\it Provable}\ \ai{G}; \] 
\[ \Ai{\lc{p}F}\ \ =\ \ \Proof{\ai{p},\ai{F}}.\]
{
\begin{theorem}\label{asoundness} {\bf (Soundness of $\gla$ with respect to arithmetical provability)} \\ 
{\it For any Constant Specification {\it CS} and any arithmetical interpretation $\ast$ respecting {\it CS}, if $\glacs\proves F$ then $\pa\proves F^\ast$.}
\end{theorem}
\begin{proof}
It is immediate that Reflection Rule is valid: 
if {\it Provable F} is derivable in $\PA$, then {\it Provable F} is true hence $F$ is provable. 

Validity of C1 and C2 immediately follows from $\Sigma$-comple\-te\-ness of $\PA$. 

Soundness of Explicit-Implicit Reflection takes place since  $\lc{t}\Box F\imp F$ is derivable from other principles of {\sf GLA}, which is already proved sound. 

\end{proof}
Arithmetical completeness of $\glae$ could be established following arithmetical completeness proofs from \cite{Art94,Art95,Art01a} (cf. also  \cite{Y_S01}). 

\begin{theorem}\label{arithmeticalcompleteness} {\bf (Arithmetic completeness)} \emph{For any finite constant specification $C\!S$, if $\glacs\not\proves
F$, then there exists a $C\!S$-interpretation $\ast$ such that $\pa\not\proves\ai{F}$.}
\end{theorem}
\begin{proof} The claim of the theorem follows from the arithmetical completeness of $\glae$. 
\end{proof}

\subsection{Explicit-Implicit  Reflection vs. Implicit-Explicit Reflection}

Explicit-Implicit Reflection $\lc{x}\Box F\imp F$, as we have seen in Theorem~\ref{asoundness}, is arithmetically valid. However, the Implicit-Explicit Reflection
\[ \mbox{\it IER}\ =\ \Box\lc{x}P\imp P \]
is not a provable principle. 

1. A proof via \gla. 

It suffices to establish that IER is not derivable in $\glae$. For this we will use an appropriate Kripke model. Take
$$\mbox{\it $W=\{1,2\}$, $1\prec 2$, $P$ is false at $1$ and $2$, ${\cal E}(t,F)$ is always false.}$$
$$\begin{array}{ll}
2 & \ \ \ \ \ \neg P,\ \neg\lc{x}P,\ \Box\lc{x}P,\ \neg(\Box\lc{x}P\imp P)\ \ \mbox{(i.e., $\neg\mbox{\it IER}$)} \\
\uparrow & \\
1 & \ \ \ \ \ \neg P,\ \neg\lc{x}P,\ \neg\Box\lc{x}P,\ \Box\lc{x}P\imp\lc{x}P\ \ \mbox{({\it IER}-soundness)}
\end{array}
$$ 
Therefore, IER is false at node 2 of the model. 

2.  An arithmetical proof. 

If $P=\bot$, then $\lc{x}P$ is provably equivalent to $\bot$. Therefore, this instance of {\it IER} is equivalent to $\Box\bot\imp\bot$, which is the consistency statement, not provable in \PA.}

For other reflection principles of \pa\ see our paper \cite{Nog14b}.

\section{Acknowledgements}
The author is grateful to Sergei Artemov, Melvin Fitting, Evan Goris, Gerhard J\"ager, Makoto Kikuchi, Taishi Kurahashi, Hidenori Kurokawa, Franco Montagna, Anil Nerode, Thomas Strahm, Thomas Studer, Tatiana Yavorskaya, Junhua Yu, Ren-June Wang, logic groups in Bern University, Nihon University in Tokyo, Kobe  University, Academia Sinica and National Chung Cheng University of Taiwan for useful discussions.

\end{document}